\documentclass[a4paper,11pt]{amsart}

\usepackage{amsmath}
\usepackage{amsfonts}
\usepackage{amssymb}
\usepackage{amsthm}
\usepackage{dsfont}
\usepackage{mathtools}
\usepackage{nicefrac,xfrac}
\usepackage{derivative}
\usepackage{enumerate}
\usepackage[backend=bibtex]{biblatex}
\newtheorem{theorem}{Theorem}[section]

\theoremstyle{remark}
\newtheorem*{remark}{Remark}

\theoremstyle{definition}

\usepackage{comment}
\usepackage{todonotes}

\newcommand{\defeq}{\vcentcolon=}
\newcommand{\diff}[1]{\,\mathrm{d}#1}
\newcommand{\abs}[1]{\left|#1\right|}
\newcommand{\set}[1]{\left\{#1\right\}}
\newcommand{\norm}[1]{\left\|#1\right\|}

\newcommand{\inner}[2]{\left(#1\cdot#2\right)}

\title[Measurable Semigroup of the Heat Flow for Harmonic Maps]{Measurable Semigroup Selection of the Heat Flow for Harmonic Maps}
\author{Jorge E. Cardona}

\begin{document}

\address{\parbox{\linewidth}{
    \textsc{Jorge E. Cardona} \\
    Friedrich-Schiller-Universität Jena \\
    Fakultät für Mathematik und Informatik, \\
    Institut für Mathematik\\
    Ernst-Abbe-Platz 2, 07743 Jena, Germany\\
  }}
\email{jorge.cardona@uni-jena.de}

%  Math Subject Classifications
%  \subjclass[2010] {Primary: 35Q99,28B20, 58E20; Secondary:35B99, 82D99}
\date{\today}

\begin{abstract}
  J.-M.~Coron proved in \cite{Coron1990NonuniquenessFT} that the global weak solutions of the heat flow from $M$ to $N$, starting at non-stationary weakly harmonic maps, are not unique when $M=B^3$ and $N=S^2$. Hence, the semigroup property of the solution map does not hold in general. The present short paper uses the techniques developed by J.~Cardona and L.~Kapitanski to show the existence of infinitely many measurable semigroups solving the heat flow in the same cases where non-uniqueness was shown by J.-M.~Coron.
\end{abstract}

\maketitle

\section{Introduction}

Let $(M, g)$ and $(N, h)$ be two Riemannian manifolds.
Following \cite[Chapter 9]{Jost2017} and \cite[Section 3.4]{Eells1978ARO}, the \textit{energy density} of a smooth map $u: M \to N$ at $x \in M$ is the quantity $$e(u)(x) = \frac12 \abs{du(x)}^2\,,$$ where $du(x)$ is the differential of $u$ at $x$, and the norm is the one in the tensor product $T^*_xM \otimes T_{u(x)}N$.
In local coordinates $(x^1, \dots, x^m)$ on $M$ and $(y^1, \dots y^n)$ on $N$ the energy density is $$e(u)(x) = g^{\alpha \beta}(x) h_{ij}(u(x)) \pdv{u^i(x)}{x^\alpha} \pdv{u^j(x)}{x^\beta}\,,$$ where the usual summation convention is used, and the Latin indices run over $1, \dots, n$ and the Greek indices over $1, \dots, m$.

The \textit{energy of the map $u$} is the quantity $$E(u) = \int_M e(f) \diff{M}\,,$$ where $\diff{M}$ is the volume element in $M$.
The Euler-Lagrange equation associated to the functional $E$ is
\begin{equation}
  \label{eq:euler-lagrange-energy}
  \Delta_M u^i(x) + g^{\alpha \beta}(x) \Gamma^i_{jk}(u(x)) \pdv{u^j(x)}{x^\alpha} \pdv{u^k(x)}{x^\beta} = 0
\end{equation}
where $\Delta_M$ is the Laplace-Beltrami operator on $M$, i.e., $$\Delta u^i(x) = \frac{1}{\sqrt{\abs{g}}} \pdv*{\left(\sqrt{\abs{g}} g^{\alpha \beta} \pdv{u^i(x)}{x^\beta}\right)}{x^\alpha}\,,$$ and $\Gamma^i_{jk}(u(x))$ are the Christoffel symbols of the metric $h$ at $u(x)$.

A map $u \in C^2(M, N)$ is said to be \textit{harmonic} if it satisfies (\ref{eq:euler-lagrange-energy}).

A central question in the study of harmonic maps is the following: Given an arbitrary smooth map $u_0: M \to N$, is it possible to deform it into an harmonic map? J.~Eells and J.H.~Sampson~\cite{Eells1964} proved that, if $\text{dim}(M) > 2$ and $N$ has a non-positive sectional curvature, the heat-flow associated to the functional $E$, i.e., the Cauchy problem
\begin{subequations}
  \label{eq:heat-flow}
  \begin{alignat}{1}
    \partial_t u - \Delta_M u & = \Gamma_N(u)(\nabla u, \nabla u)_M  \\
    u(0, x) &= u_0
  \end{alignat}
\end{subequations}
has a global and regular solution $u(t, x)$ which converges to a harmonic map as $t \to \infty$ (where $\Gamma_N(u)(\nabla u, \nabla u)_M$ is the second term in (\ref{eq:euler-lagrange-energy}).)
M.~Struwe~\cite{Struwe1985} considered global weak solutions of~(\ref{eq:heat-flow}) to extend the results of Eells and Sampson to arbitrary $N$ when $\operatorname{dim}(M) = 2$.

A map $u \in H^1_\text{loc}(M, N)$ is said to be \textit{weakly harmonic} if it satisfies (\ref{eq:euler-lagrange-energy}) in the weak sense, i.e.,
$$\int_M g^{\alpha \beta}(x) h_{ij}(u(x)) \pdv{u^i(x)}{x^\alpha}\left(\pdv{\eta^j(x)}{x^\beta} - \eta^k(x) \pdv{u^l(x)}{x^\beta} \Gamma^j_{lk}(u(x))\right) \diff{M} = 0$$ for every smooth map $\eta: M \to TN$ such that $\eta(x) \in T_{u(x)}N$.

Y.~Chen \cite{Chen1989TheWS} proved the existence of global weak solution of (\ref{eq:heat-flow}) when $N$ is a sphere and $M$ is a compact and smooth manifold without boundary and with $\operatorname{dim}(M) > 2$. J.-M. Coron~\cite{Coron1990NonuniquenessFT} and later F. Béthuel, J.-M. Coron, J.-M. Ghidaglia, and A. Soyeur \cite{Bethuel1992} proved the existence of infinitely many global weak solutions for some initial conditions in the case $M = B^3$ and $N=S^2$. In particular, the initial conditions they consider are weakly harmonic maps that fail to be stationary points of the energy functional. A map $u$ is stationary if it is invariant under variations on the spatial domain, i.e.,
$$\left. \frac{d}{d\varepsilon} E(u_{\varepsilon\eta}) \right|_{\varepsilon = 0} = 0 \text{ for every smooth vector field } \eta$$
where $u_{\varepsilon\eta}(x) = u(\operatorname{exp}_x(\varepsilon \eta(x)))$ and $\operatorname{exp}_x$ is the exponential map on $M$ at $x$.

The non-uniquenes result in \cite{Bethuel1992} begs the question: Is there a semigroup solving the heat-flow? In general, if the solutions of an evolution equation are unique the existence of a semigroup is warranted. Due to the non-uniqueness result in the case $M=B^3$ and $N=S^2$, is not immediately obvious that such a semigroup exists. The main result in this paper is the existence of an infinite number of semigroups solving the heat flow in the same case considered in \cite{Bethuel1992}.  The main tool is the measurable semigroup selection thereom developed by J.~Cardona and L.~Kapitanski~\cite{Cardona2020} in the same spirit of the Markov selection theorem of N.V.~Krylov~\cite{Krylov1973}.

\subsection{Notation}

In what follows, $M = B^3 \defeq \set{x \in \mathbb{R}^3 : \abs{x} < 1}$ and $N= \partial B^3 \subset \mathbb{R}^3$, both are considered as sub-manifolds of $\mathbb{R}^3$. A map $u: M \to N$ is simply a map $u: B^3 \to \mathbb{R}^3$ with the constraint $\abs{u(x)} = 1$ for every $x \in B^3$.
The energy of a smooth map $u: M \to N$ is the quantity $$E(u) = \frac12 \norm{u}^2_{H^1} = \frac12 \sum_{i=1}^3 \int_{B^3} \abs{\nabla u^i(x)}^2 \diff{x}\,,$$ and the space of measurable functions with finite energy is denoted by $H^1(M,N)$ (or simply $H^1$) endowed with the norm $\norm{\cdot}_{H^1}$.
Since $\partial B^3$ is a $C^1$ surface, the trace operator $T: H^1 \to L^2(\partial B^3)$ is bounded and linear.

A map $u \in C^2$ is \textit{harmonic} if and only if it satisfies
\begin{equation}
  \label{eq:harm-simple}
  \Delta u + u \abs{\nabla u}^2 = 0\,.
\end{equation}
A map $u \in H^1$ is weakly harmonic if and only if (\ref{eq:harm-simple}) holds in the weak sense.
A weakly harmonic map $u$ is stationary if and only if $$\sum_{j=1}^3 \pdv*{\abs{\pdv{u(x)}{x^j}}^2}{x^k} - 2  \pdv*{\inner{\pdv{u(x)}{x^j}}{\pdv{u(x)}{x^k}}}{x^j} = 0 \quad \forall k=1,2,3\,.$$

Functions from $[0, \infty) \times B^3$ to $S^2$ are seen as paths from $[0, \infty)$ taking values in some functional space. The space of paths taking values in $H^1$ with uniform bounds is denoted by $L^\infty([0, \infty); H^1)$. The space of paths taking values in $L^2$ and continuous with respect to its strong topology is denoted by $C([0, \infty); L^2)$. Finally, the space of paths taking values in $H^1$ and continuous with respect to its weak topology is denoted by $C_w([0, \infty); H^1)$.

For the definitions and notation of set-valued analysis we refer to \cite{Cardona2020} and reference therein.

\section{Measurable semigroups of the heat-flow}

The heat flow reads
\begin{subequations}
  \label{eq:heat-flow-simple}
  \begin{alignat}{1}
    \partial_t u - \Delta u & = u \abs{\nabla u}^2 \\
    u(0) &= u_0
  \end{alignat}
\end{subequations}

Following \cite{Chen1989TheWS} and \cite{Bethuel1992}, a \textit{global weak solution} of (\ref{eq:heat-flow-simple}), started at $u_0 \in H^1$, is a measurable map $u$ from $[0, \infty) \times B^3$ to $\mathbb{R}^3$ that satisfies the following
\begin{enumerate}[i)]
\item $u$ takes values in $S^2$, i.e. $\abs{u(t, x)} = 1$ for a.e. $(t,x) \in [0, \infty) \times M$;
\item $u$ satisfies the initial condition $u(0) = u_0$ and the boundary condition $T u(t) = T u_0$ for every $t \geq 0$;
\item $u \in L^\infty([0, \infty); H^1) \cap C([0, \infty); L^2) \cap C_w([0, \infty); H^1)$;
\item the equation (\ref{eq:heat-flow-simple}) holds in the weak sense; and
\item the following energy inequality holds $$E(u(t + s)) + \int_t^{t+s} \int_M \abs{\partial_t u(\tau)}^2 \diff{M} \diff{\tau} \leq E(u(t))$$ for every $s \geq 0$ and almost every $t \geq 0$ including $t=0$.
\end{enumerate}

\begin{remark}
  Note that v) is usually written only with $t=0$. In general, it is not possible to obtain an inequality valid for every $t \geq 0$ due to the mode of convergence of the limit arguments in the constructions of solutions in \cite[Eqs. 2.19-2.22]{Chen1989TheWS} or \cite[Eqs. 1.8,1.9, and 1.12]{Bethuel1992}. See \cite[Section 4.4]{Cardona2020} for a similar situation with the Navier-Stokes system.
\end{remark}

\begin{theorem}\label{thm:main}
  There exist infinitely many measurable maps $$u: H^1 \to C_w([0, \infty), H^1)$$ such that $u(a)$ is a global weak solutions of (\ref{eq:heat-flow-simple}) started at $a \in H^1$, such that $u(a, 0) = a$ and $u(a, t + s) = u(u(a, t), s)$ for every $s \geq 0$ and almost every $t \geq 0$ (including $t = 0$).
\end{theorem}

\begin{proof}
  For every $a \in H^1$, let $S_a \subset C_w([0, \infty); H^1)$ be the set of all the weak global solutions of~(\ref{eq:heat-flow-simple}). We know this set is not empty from Theorem 1 in \cite{Bethuel1992}. The arguments to pass to the limits in Theorem 1 in \cite{Bethuel1992} ensure that the set-valued map $a \to S_a$ is upper-semicontinuous, hence measurable. Moreover, the same arguments ensure that the set $S_a$ is compact in the topology of $C_w([0, \infty); H^1)$. So, the set-valued map $a \mapsto S_a$ is valued in the non-empty compact subsets of $C_w([0, \infty); H^1)$.

  Let $u \in S_a$, and let $t$ be such that the energy inequality is valid. Is not hard to see that for every $v \in S_{u(t)}$ the map $$w(s) = \begin{cases}u(s) & \text{ for } s \leq t \\ v(s - t) & \text{ for } s > t\end{cases}$$ is an element in $S_a$. Theorem 2.5 in \cite{Cardona2020} ensures the existence of a measurable semigroup.

  Finally, we need to show that Theorem 2.5 in \cite{Cardona2020} yield infinitely many measurable semigroup. Let $a \in H^1$ be a weakly harmonic map that is not stationary. It was shown in \cite{Bethuel1992} that (\ref{eq:heat-flow-simple}) has at least two global weak solutions $u_1$ and $u_2$, moreover, one of them is constant $u_1(t) = a$.
Recall that the weak topology of $H^1$ is metrizable on the closed ball of radius $2\norm{a}$, hence, there is a weakly continuous function $\varphi: H^1 \to \mathbb{R}$ satisfying $\varphi(a) = \max \set{ \varphi(v) : v \in H^1 \text{ with } \norm{v} \leq \norm{a}}$. Since the two continuous maps $t \mapsto \varphi(u_1(t))$ and $t \mapsto \varphi(u_2(t))$ are different, there exists $\lambda > 0$ such that $I_{\lambda, \varphi}[u_1] > I_{\lambda, \varphi}[u_2]$ where $$I_{\lambda, \varphi}[u] = \int_0^\infty e^{- \lambda t } \varphi(u(t)) \diff{t}\,.$$ By ensuring that the first functional used to refine the set-valued maps in the proof of Theorem 2.5 in \cite{Cardona2020} is either $I_{\lambda, \varphi}$ or $I_{\lambda, -\varphi}$ the resulting semigroup would be different. Hence, different enumerations of the family of separating functions used in \cite{Cardona2020} results in different semigroups.
\end{proof}

\subsection{An Example: Landau-Lifshitz equations}

The system of Landau-Lifshitz equations $$\partial_t u = u \times \Delta u - \lambda u \times (u \times \Delta u)\quad \abs{u} = 1\,$$ describing the evolution of spin fields in continuum ferromagnetism enjoys similar bounds and it was proven by F. Alouges and A. Soyeur \cite{Alouges1992} that weak solutions are also non-unique using the same method as J.-M. Coron \cite{Coron1990NonuniquenessFT}. Theorem \ref{thm:main} ensures the existence of infinitely many measurable semigroups solving the Landau-Lifshitz equations.

\printbibliography

\end{document}